\documentclass[12 pt]{article}%
\usepackage{amsmath, amsfonts, amsthm, color,latexsym}
\usepackage{amsmath, ulem}
\usepackage{amsfonts}
\usepackage{amssymb}
\usepackage{color, soul}
\usepackage[all]{xy}
\usepackage{graphicx}%
\providecommand{\U}[1]{\protect\rule{.1in}{.1in}}
\allowdisplaybreaks[4]
\newtheorem{theorem}{Theorem}[section]
\newtheorem{proposition}[theorem]{Proposition}
\newtheorem{corollary}[theorem]{Corollary}
\newtheorem{example}[theorem]{Example}
\newtheorem{examples}[theorem]{Examples}

\newtheorem{final remark}[theorem]{Final Remark}
\newtheorem{definition}[theorem]{Definition}
\allowdisplaybreaks[4]

\begin{document}

\title{Infinite dimensional spaces consisting of sequences that do not converge to zero}
\author{Mikaela Aires\thanks{Supported by a CNPq scholarship}~~and Geraldo Botelho\thanks{Supported by FAPEMIG grants RED-00133-21 and APQ-01853-23. \newline 2020 Mathematics Subject Classification: 46B42, 46B87, 47B07, 47H60.\newline Keywords: Banach sequence spaces, Banach sequence lattices, almost pointwise spaceability, vector topology. }}
\date{}
\maketitle

\begin{abstract} Given a map $f \colon E \longrightarrow F$ between Banach spaces (or Banach lattices), a set $A$ of $E$-valued bounded sequences, ${\bf x} \in A$ and a vector topology $\tau$ on $F$, we investigate the existence of an infinite dimensional Banach space (or Banach lattice) containing a subsequence of ${\bf x}$ and consisting, up to the origin, of sequences $(x_j)_{j=1}^\infty$ belonging to $A$ such that $(f(x_j))_{j=1}^\infty$ does not converge to zero with respect to $\tau$. The applications we provide encompass the improvement of known results, as well as new results, concerning Banach spaces/Banach lattices not satisfying classical properties and linear/nonlinear maps not belonging to well studied classes.
%
%
\end{abstract}

\section{Introduction}

One of the problems in the field of lineability (see \cite{Aron}) is the following: given a subset $A$ of a Banach space (or a topological vector space) $E$, is there a closed infinite dimensional subspace $W$ of $E$ contained in $A \cup \{0\}$? If yes, the set $A$ is said to be {\it spaceable}. Many proofs start with a vector $x \in A$ and then manipulate $x$ conveniently to construct the subspace $W$. Sometimes, the mother vector $x$ does not belong to $W$. In \cite{Pellegrino}, the set $A$ is said to be {\it pointwise spaceable} if, regardless of the vector $x \in A$, there is a closed infinite dimensional subspace $W$ of $E$ contained in $A \cup \{0\}$ and containing $x$. Pointwise lineability was further developed in, e.g., \cite{bernal, calderon, viniciuspams, liu, anselmogeivison}. In the setting of sequence spaces, sometimes one can get a bit more than spaceability and a bit less than pointwise lineability: maybe that, given a sequence $x \in A$, one can find a closed infinite dimensional subspace $W$ of $E$ contained in $A \cup \{0\}$ and containing {\it a subsequence} of $x$. In this case, we shall say that the set $A$ is {\it almost pointwise spaceable}.

 The motivation to study this concept is the fact that, quite often, it is known that a set is spaceable but not if it is (almost) pointwise spaceable. Let us give an illustrative example: On the one hand, it is easy to see that the set
$${\cal C} = \left\{ (a_j)_{j=1}^\infty \in \ell_\infty : a_{2j-1}\!\cdot\!a_{2j}= 0 \mbox{ for every } j \mbox{ and } a_{2j-1} + a_{2j} \not\longrightarrow 0 \right\} $$
is spaceable (of course, $a_j \not\longrightarrow 0$ means that the sequence of numbers $(a_j)_{j=1}^\infty$ does not converge to zero). On the other hand, it seems to us that it is not obvious whether or not $\cal C$ is (almost) pointwise spaceable. As an application of our main theorem, using a map that is neither linear nor a homogeneous polynomial, we shall show in Example \ref{exk} that $\cal C$ is almost pointwise spaceable.

The purpose of this paper is to prove a general theorem (cf. Theorem \ref{main}) that gives conditions on a (non necessarily linear) map $f \colon E \longrightarrow F$ between Banach spaces, on a subset $A$ of $\ell_\infty(E)$ and on a vector topology $\tau$ on $F$ so that the set of $E$-valued sequences $(x_j)_{j=1}^\infty$ belonging to $A$ such that $(f(x_j))_{j=1}^\infty$ does not converge to zero with respect to $\tau$ is almost pointwise spaceable in $\ell_\infty(E)$. If $E$ is a Banach lattice, then given any positive sequence ${\bf x} \in A$, the general theorem assures  the existence of an infinite dimensional closed {\it sublattice} of $\ell_\infty(E)$ contained in $A \cup \{0\}$ and containing a subsequence of ${\bf x}$ (cf. Theorem \ref{ret}). The proof of our general theorem is a refinement of a technique introduced by Jim\'enez-Rodr\'iguez \cite{Jimenez}.

In Section 3, devoted to applications to Banach space theory, we obtain corollaries about non-completely continuous operators; Banach spaces failing the Schur property; non-weakly sequentially continuous polynomials; Banach spaces failing the polynomial Schur property; sets of weak$^*$-null non-norm null sequences in dual Banach spaces; non-$p$-convergent operators; Banach spaces failing the Grothendieck property; and non-pseudo weakly compact operators.

In Section 4, devoted to applications to Banach lattice theory, we prove corollaries about sets of order bounded (or norm bounded), disjoint, non-norm null sequences; Banach lattices failing the following properties: disjoint Grothendieck property, positive Schur property, dual positive Schur property, positive Grothendieck property; and to operators not belonging to the following classes:  order weakly compact operators, $M$-weakly compact operators, almost Dunford-Pettis operators, and weak $M$ weakly compact operators.

Some of the applications above improve known results and some are the first lineability-type results obtained to their specific situations.

Throughout the paper, Banach spaces are real or complex and Banach lattices are always real. The topological dual of a Banach space/lattice $E$ shall be denoted by $E^*$. By an operator we mean a bounded linear operator, and $T^*$ stands for the adjoint of the operator $T$. For the language and notation of Banach space theory, we refer to \cite{livro, megginson}. For topological vector spaces, see \cite[Chapter 8]{livro}; for homogeneous polynomials, see \cite{dineen, mujica}; for Banach lattices, see \cite{Ali, Ali1, Meyer}.

\section{Main result and first consequences}

As mentioned in the Introduction, we introduce a concept which, in spaces consisting of Banach-valued sequences, lies between spaceability and pointwise spaceability. Let $E$ be a Banach space. A linear subspace of $E^{\mathbb{N}}$ endowed with a complete norm shall be called a {\it Banach sequence space}. Examples: (i) The spaces $c_0(E)$ of $E$-valued norm null sequences, $c_0^w(E)$ of $E$-valued weakly null sequences, and $\ell_\infty(E)$ of bounded $E$-valued sequences, are Banach sequence spaces with the supremum norm. (ii) For $1 \leq p < \infty$, the spaces $\ell_p(E)$ of absolutely $p$-summable $E$-valued sequences and $\ell_p^w(E)$ of weakly $p$-summable $E$-valued sequences are Banach sequence spaces endowed with their natural norms (see \cite{diestel+jarchow+tonge}). In this paper we are interested in Banach sequence spaces consisting of sequences which do not converge to zero in certain topologies.

\begin{definition}\rm A nonvoid subset $A$ of a Banach sequence space $X$ is said to be {\it almost pointwise spaceable} if, for every sequence ${\bf x} \in A$, there exists a closed infinite dimensional subspace of $X$ contained in $A \cup \{0\}$ and containing a subsequence of ${\bf x}$.
\end{definition}


\begin{definition} \rm A map $f \colon E \longrightarrow F$ between topological vector spaces is said to be of {\it homogeneous type} if $f$ is continuous at $0$ and there exists a nonzero integer number $n$ such that $f(\lambda x) = \lambda^n f(x)$ for every scalar $\lambda \neq 0$ and every $x \in E$.

It is easy to see that $f(0) = 0$ in this case.
\end{definition}

\begin{example} \rm (a) Continuous linear operators between topological vector spaces are maps of homogeneous type.\\
(b) A map $P \colon E \longrightarrow F$ between topological vector spaces is a {\it homogeneous polynomial} if there are $n \in \mathbb{N}$ and an $n$-linear operator $A \colon E^n \longrightarrow F$ such that $P(x) = A(x, \ldots, x)$ for every $x \in E$ (see \cite{dineen}). Continuous homogeneous polynomials are maps of homogeneous type.\\
(c) A quite useful map of homogeneous type, which is neither a linear operator nor a homogenous polynomial, shall be presented in Example \ref{exk}.
\end{example}

\begin{definition}\rm  Let $E$ be a Banach space. A subset $A$ of $\ell_\infty(E)$ is said to be:\\
(i) {\it Subsequence invariant} if subsequences of sequences belonging to $A$ belong to $A$ as well. \\
(ii) {\it $\ell_\infty$-complete} if, for $(x_j)_{j=1}^\infty \in A$ and $(\alpha_j)_{j=1}^\infty \in \ell_\infty$, it holds $(\alpha_j x_j)_{j=1}^\infty \in A$.
\end{definition}

Plenty of subsequence invariant and $\ell_\infty$-complete sets shall appear along the paper.

A topology $\tau$ on a linear space shall be called a {\it vector topology} if $(E, \tau)$ is a topological vector space. In this case, a sequence $(x_j)_{j=1}^\infty$ in $E$ that converges to zero with respect to $\tau$ shall be called a $\tau$-null sequence, in symbols, $x_j \stackrel{\tau}{\longrightarrow} 0$. Otherwise, the sequence is called non-$\tau$-null, in symbols $x_j \stackrel{\tau}{\not\longrightarrow} 0$. For the norm topology on a normed space we simply write $x_j \longrightarrow 0$ and $x_j \not\longrightarrow 0$.

Given two topologies $\tau$ and $\tau'$ on the same set, we say that $\tau$ is {\it weaker} than $\tau'$ if $\tau \subseteq \tau'$. The weak topology on a normed space and the weak$^*$ topology on a dual space shall be denoted by $\omega$ and $\omega^*$, respectively.

As pointed out in the Introduction, the proof of our main result is a refinement of an argument due to Jim\'enez-Rodr\'iguez \cite{Jimenez}.

\begin{theorem} \label{main} Let $f \colon E \longrightarrow F$ be a map of homogeneous type between Banach spaces, let $A$ be a subsequence invariant and $\ell_\infty$-complete subset of $\ell_\infty(E)$, and let $\tau$ be a vector topology on $F$ weaker than the norm topology. Then the subset $\cal C$ of sequences $(x_j)_{j=1}^\infty$ in $A$ for which $(f(x_j))_{j=1}^\infty$ is non-$\tau$-null is empty or almost pointwise spaceable in $\ell_\infty (E)$.

 Moreover,
${\cal C} \cap c_0^w(E)$ is empty or almost pointwise spaceable in $c_0^w(E)$.
\end{theorem}

\begin{proof} Suppose that $\mathcal{C}$ is nonempty and pick $(z_j)_{j=1}^\infty \in {\cal C}$.  Since $f(z_j) {\stackrel{\tau}{\not\longrightarrow}} 0$ in $F$, by a well known fact about convergent sequences in topological spaces there exists a subsequence $(y_j)_{j=1}^\infty$ of $(z_j)_{j=1}^\infty$ so that
\begin{equation}\label{eqqq} \mbox{no subsequence of } (f(y_j))_{j=1}^\infty \mbox{ converges to } 0 \mbox{ with respect to the topology }\tau.
 \end{equation}
 In particular, $f(y_j) {\stackrel{\tau}{\not\longrightarrow}} 0$, hence $\|f(y_j)\|\not\longrightarrow 0$ because $\tau$ is weaker than the norm topology. Then there are $\varepsilon > 0$ and a subsequence $(x_j)_{j=1}^\infty$ of $(y_j)_{j=1}^\infty$ such that $\|f(x_j)\|\geq \varepsilon$ for every $j$. By the continuity of $f$ in $0$, there exists $\delta > 0$ so that $\|f(x)\| = \|f(x) - f(0)\| < \varepsilon$ for every $x \in E$ with $\|x\| < \delta$. It follows that  $\|x_j\|\geq \delta$ for every $j$. Since $(x_j)_{j=1}^\infty$ is a subsequence of $(z_j)_{j=1}^\infty \in A$ and $A$ is subsequence invariant, we have  $(x_j)_{j=1}^\infty \in A$. As $(f(x_j))_{j=1}^\infty$ is a subsequence of $(f(y_j))_{j=1}^\infty$, from (\ref{eqqq}) we get that  $f(x_j) {\stackrel{\tau}{\not\longrightarrow}} 0$. This proves that $(x_j)_{j=1}^\infty \in {\cal C}. $

Let us consider the set of prime numbers $\{p_k: k \in \mathbb{N}\}$ increasingly ordered and the surjective map $G \colon \mathbb{N} - \{1\} \longrightarrow \mathbb{N}$ given by
$$ G(q) = k,\mbox{ where } p_k = \min\{p : p \mbox{ is prime and divides } q\}.$$
We also consider the map
$$T \colon \ell_\infty \longrightarrow \ell_\infty(E)~,~T((a_j)_{j=1}^\infty) = (a_{G(j+1)}x_j)_{j=1}^\infty.  $$
Since $(a_{G(j+1)})_{j=1}^\infty \in \ell_\infty$, $(x_j)_{j=1}^\infty \in A$ and $A$ is $\ell_\infty$-complete, $T(\ell_\infty) \subseteq A \subseteq \ell_\infty(E)$; in particular, $T$ is well defined. Note that, in the case that $(z_j)_{j=1}^\infty \in c_0^w(E)$, we have $x_j \stackrel{\omega}{\longrightarrow} 0$ (because sequences of weakly null sequences are weakly null) and  $(a_{G(j+1)})_{j=1}^\infty \in \ell_\infty$. Then, for every $x^* \in E^*$,
$$x^*(a_{G(j+1)}x_j) = a_{G(j+1)}x^*(x_j) \longrightarrow 0, $$
which proves that $T((a_j)_{j=1}^\infty)\in c_0^w(E)$. So, $T(\ell_\infty) \subseteq c_0^w(E)$ in the case that ${\cal C} \cap c_0^w(E) \neq \emptyset$.

It is plain that $T$ is linear. As $(x_j)_{j=1}^\infty \in A \subseteq \ell_\infty(E)$, there is $L > 0$ so that $\|x_j\|\leq L$ for every $j$. For every $(a_j)_{j=1}^\infty \in \ell_\infty$, using that $G$ surjective we have $\{|a_j|: j \in \mathbb{N}\} = \{|a_{G(j+1)}|: j \in \mathbb{N}\}$, from which it follows that
\begin{align*}\delta \|(a_j)_{j=1}^\infty\|_\infty & = 
 \sup_j \delta|a_j| = \sup_j \delta |a_{G(j+1)}| \leq \sup_j |a_{G(j+1)}|\cdot\|x_j\|  = \sup_j\|a_{G(j+1)}x_j\|\\& = \|T(a_j)_{j=1}^\infty \|_\infty = \sup_j |a_{G(j+1)}|\cdot\|x_j\| \leq L\cdot \sup_j|a_j| = L \|(a_j)_{j=1}^\infty\|_\infty.
\end{align*}
We have just proved that $T$ is an isomorphism into; in particular $T(\ell_\infty)$ is a closed subspace of $\ell_\infty(E)$ (of $c_0^w(E)$ in the case that ${\cal C} \cap c_0^w(E) \neq \emptyset$) isomorphic to $\ell_\infty$. Since $T((1,1,1,1,\ldots)) = (x_j)_{j=1}^\infty$, $T(\ell_\infty)$ contains a subsequence  of the original sequence $(z_j)_{j=1}^\infty$.

All that is left to prove is that $T(\ell_\infty) \subseteq {\cal C} \cup \{0\}$. We already know that $T(\ell_\infty) \subseteq A$.
Let $0 \neq \xi \in T(\ell_\infty)$ be given, say $\xi = T((a_j)_{j=1}^\infty) = (a_{G(j+1)}x_j)_{j=1}^\infty$ for some sequence $(a_j)_{j=1}^\infty \in \ell_\infty$. As $\xi \neq 0$, there is $j$ such that $a_{G(j+1)} \neq 0$. Calling $p$ the smallest prime number that divides $j+1$ and $m$ the position of $p$ at the list of prime numbers, we have
$$m = G(j+1) = G(p) = G(p^2) = G(p^3) = \cdots = G(p^k) = \cdots  .$$
Therefore, $0 \neq a_{G(j+1)}= a_m = G(p^k)$ for every $k$. Then $(f(a_m x_{p^k-1}))_{k=1}^\infty = (f(a_{G(p^k)} x_{p^k-1}))_{k=1}^\infty$ is a subsequence of $(f(a_{G(j+1)}x_j))_{j=1}^\infty$. Suppose that $f(a_m x_{p^k-1}) \stackrel{\tau}{\longrightarrow} 0$. Calling $n$ the nonzero integer number that works for $f$ in definition of maps of homogeneous type, using that $\tau$ is a vector topology and that $a_m \neq 0$, we get
$$f(x_{p^k-1}) =  f\left(\frac{a_m}{a_m}x_{p^k-1}\right) = \frac{1}{a_m^n} f(a_m x_{p^k-1}) \stackrel{\tau}{\longrightarrow} \frac{1}{a_m^n}\cdot 0 = 0.   $$
Noting that $(x_{p^k-1})_{k=1}^\infty$ is a subsequence of $(x_j)_{j=1}^\infty$, which, in its turn, is a subsequence of $(y_j)_{j=1}^\infty$, it follows that $(f(x_{p^k-1}))_{k=1}^\infty$ is a subsequence of $(f(y_j))_{j=1}^\infty$. By (\ref{eqqq}), the convergence $f(x_{p^k-1}) \stackrel{\tau}{\longrightarrow} 0$ does not hold. This contradiction yields that $(f(a_m x_{p^k-1}))_{k=1}^\infty$ is a non-$\tau$-null sequence. Since $(f(a_m x_{p^k-1}))_{k=1}^\infty $ is a subsequence of $(f(a_{G(j+1)}x_j))_{j=1}^\infty$, we have that $(f(a_{G(j+1)}x_j))_{j=1}^\infty$ is a non-$\tau$-null sequence as well. This proves that $\xi \in {\cal C}$ and completes the proof.
\end{proof}

It is worth noting that, in the proof above, we had to pass to a subsequence in order to get a seminormalized sequence, which is a necessary step for the operator $T$ to be an isomorphism into. So, if we can pick a seminormalized sequence at the begining of the proof, then the whole sequence, and not only a subsequence of it, shall belong to the resulting space $T(\ell_\infty)$. This remark leads to the following consequence of the proof.

\begin{corollary}\label{ncor} Let $f, A, \tau$ and $\cal C$ be as in Theorem {\rm \ref{main}}. If there is a seminormalized sequence $\bf x$ belonging to $\cal C$, then there exists a closed infinite dimensional subspace of $\ell_\infty(E)$ (or $c_0^w(E)$) contained in ${\cal C} \cup \{0\}$ and containing $\bf x$.
\end{corollary}

Every consequence of Theorem \ref{main} has a second statement via Corollary \ref{ncor} concerning an initial seminormalized sequence. For simplicity, this second statement shall be explicitly stated only when the underlying map $f$ is the identity operator.

Next we apply the main result, using a map of homogeneous type which is neither a linear operator nor a homogeneous polynomial, to solve the problem stated in the Introduction.

\begin{example}\label{exk}\rm The point is to prove that the set
$${\cal C} = \{ (a_j)_{j=1}^\infty \in \ell_\infty : a_{2j-1}\!\cdot\!a_{2j}= 0 \mbox{ for every } j \mbox{ and } a_{2j-1} + a_{2j} \not\longrightarrow 0\} $$
is almost pointwise spaceable in $\ell_\infty$. Let us see that the map $f\colon \mathbb{R}^2\longrightarrow \mathbb{R}^2$ given by
\begin{align*}
f(x,y)= \left \{ \begin{array}{cl}
(0,y), ~ \hbox{if} ~ x=0,  \\
(y,x), ~ \hbox{if} ~ x \neq 0,
\end{array}\right.
\end{align*}
is of homogeneous type (with $n = 1$): for $\lambda \neq 0$ and $(x,y) \in \mathbb{R}^2$,
\begin{align}\label{7ude}
f(\lambda(x,y))&=\left \{ \begin{array}{cl}
(0,\lambda y), & \hbox{if} ~ x=0,  \\
(\lambda y, \lambda x), & \hbox{if} ~ x \neq 0,
\end{array}\right.
= \lambda\left \{ \begin{array}{ccc}
(0,y), ~ \hbox{if} ~ x=0,  \\
(y,x), ~  \hbox{if} ~ x \neq 0,
\end{array}\right. = \lambda f(x,y).\end{align}
Denoting by $\|\cdot\|_\infty$ the maximum norm on $\mathbb{R}^2$ and using that $f(0,0) = (0,0)$, the inequality $\|f(x,y)\|_\infty \leq \|(x,y)\|_\infty$ gives the continuity of $f$ at $(0,0)$. It is easy to see that $f$ is nonlinear: $f((0,1)+(1,1))=(2,1)\neq (1,2)=f(0,1)+f(1,1)$. By (\ref{7ude}), $f$ is not an $n$-homogeneous polynomial for any $n \geq 2$ \cite[Ex.\,2.C, p.\,16]{mujica}.

It is clear that the set
$$A=\left\{{\big (}(a_j,b_j){\big )}_{j=1}^\infty \in \ell_\infty(\mathbb{R}^2) \colon a_jb_j= 0 \mbox{ for every }j\right\}$$
is subsequence invariant and $\ell_\infty$-complete, and that the set
$$\mathcal{D}=\left\{{\big(}(a_j,b_j){\big )}_{j=1}^\infty \in A \colon f(a_j,b_j) \not\longrightarrow (0,0) \right\}$$
is nonempty. By Theorem \ref{main}, $\cal D$ is almost pointwise spaceable in $\ell_\infty(\mathbb{R}^2)$. 
For any sequence ${\big(}(a_j,b_j){\big )}_{j=1}^\infty \in A$, it holds 
\begin{equation}\label{lm3s}a_j + b_j \longrightarrow 0 \Longleftrightarrow f(a_j , b_j) \longrightarrow (0,0). \end{equation}
Indeed, noting that $a_j = 0$ or $b_j = 0$ for every $j$, we have
\begin{align*}
f(a_j,b_j)= \left \{ \begin{array}{cl}
(0,b_j), ~\hbox{if} ~ a_j=0  \\
(0,a_j), ~\hbox{if} ~ a_j \neq 0
\end{array}\right. = (0, a_j + b_j)
\end{align*}
for every $j$, from which (\ref{lm3s}) follows. Therefore, the set
$${\cal D} =\left\{{\big (}(a_j,b_j){\big )}_{j=1}^\infty \in \ell_\infty(\mathbb{R}^2) \colon a_jb_j= 0 \mbox{ for every }j \mbox{ and } a_j + b_j \not\longrightarrow 0 \right\}$$
is almost pointwise spaceable in $\ell_\infty(\mathbb{R}^2)$. Consider now the isometric isomorphism
$$T \colon \ell_\infty \longrightarrow \ell_\infty(\mathbb{R}^2)~,~T\left((a_j)_{j=1}^\infty \right)= {\big (}(a_{2j-1},a_{2j}){\big )}_{j=1}^\infty. $$
Given a sequence ${\bf x} \in {\cal C}$, we have $T({\bf x}) \in {\cal D}$, so there is a closed infinite dimensional subspace $W$ of $\ell_\infty(\mathbb{R}^2)$ containing a subsequence of $T({\bf x})$ and $W \subseteq {\cal D} \cup \{0\}$. Then, $T^{-1}(W)$ is a closed infinite dimensional subspace of $\ell_\infty$ containing a subsequence of $\bf x$ and $T^{-1}(W) \subseteq {\cal C} \cup \{0\}$. This proves that $\cal C$ is almost pointwise spaceable in $\ell_\infty$.
%
%
\end{example}

In the following sections we shall provide many more concrete applications of the main theorem, as well as of its consequences we will prove next.

Before proceeding, this is the moment to say that, unlike the case of spaceability, pointwise spaceability and almost pointwise spaceability do not pass (automatically) from a set to its supersets. For instance, $c_0$ is pointwise spaceable in $\ell_\infty$ whereas $c_0 \cup \{(1,1,1,1, \ldots)\}$ is not even almost pointwise spaceable in $\ell_\infty$.

Recall that a subset $A$ of a linear space is {\it balanced} if $\lambda A \subseteq A$ for every scalar $\lambda$ with $|\lambda| \leq 1$.

\begin{proposition} \label{prop4}
Let $f \colon E \longrightarrow F$ be a map of homogeneous type between Banach spaces, let $\tau_E$ be a vector topology on $E$ and let $\tau_F$ be a vector topology on $F$ weaker than the norm topology. Then: \\
\noindent {\rm (a)} The set $$\mathcal{C}_1=\{(x_j)_{j=1}^\infty \in \ell_\infty(E) \colon x_j \stackrel{\tau_E}{\longrightarrow} 0 ~\textrm{and}~ f(x_j)\stackrel{\tau_F}{\not \longrightarrow} 0\}$$
is empty or almost pointwise spaceable in $\ell_\infty(E)$. Moreover, ${\cal C}_1 \cap c_0^w(E)$ is empty or almost pointwise spaceable $c_0^w(E)$.

\noindent {\rm (b)} If, in addition, $\tau$ is a vector topology on $F$, then the set
$$\mathcal{C}_2=\{(x_j)_{j=1}^\infty \in \ell_\infty(E) \colon x_j \stackrel{\tau_E}{\longrightarrow} 0, ~ f(x_j) \stackrel{\tau}{\longrightarrow}0 \textrm{ and}~ f(x_j)\stackrel{\tau_F}{\not \longrightarrow} 0\}$$
is empty or almost pointwise spaceable in $\ell_\infty(E)$. Moreover, ${\cal C}_2 \cap c_0^w(E)$ is empty or almost pointwise spaceable in $c_0^w(E)$.
\end{proposition}

\begin{proof} (a) The set $A:=\{(x_j)_{j=1}^\infty \in \ell_\infty (E) \colon  x_j \stackrel{\tau_E}{\longrightarrow} 0 \}$ is obviously subsequence  invariant. By Theorem \ref{main} it is enough to show that $A$ is $\ell_\infty$-complete. 
To do so, let $(a_j)_{j=1}^\infty \in \ell_\infty$ and $(x_j)_{j=1}^\infty \in A$ be given, and let $U$ be a neighborhood of $0$ with respect to $\tau_E$.  If $a_j=0$ for each $j \in \mathbb{N}$, then there is nothing to do  
because $0 \in U$. So, we can assume that 
$\|(a_k)_{k=1}^\infty\|_\infty\neq 0$. 
Since $x_j \stackrel{\tau_E}{\longrightarrow} 0$ and $\tau_E$ is a vector topology, we have $\|(a_k)_{k=1}^\infty\|_\infty x_j \stackrel{\tau_E}{ \longrightarrow} 0$. Since the origin of any topological vector space admits a neighborhood basis consisting of balanced sets  \cite[Proposition 8.1.7]{livro}, there exists a balanced neighborhood $V$ of $0$ such that $V \subseteq U.$ The convergence $\|(a_k)_{k=1}^\infty\|_\infty x_j \stackrel{\tau_E}{ \longrightarrow} 0$ gives a natural number $j_0 \in \mathbb{N}$ so that $\|(a_k)_{k=1}^\infty\|_\infty x_j \in V$ for every $j \geq j_0$. Using that $V$ is balanced and that $\frac{|a_j|}{\|(a_k)_{k=1}^\infty\|_\infty}\leq 1$ for any $j$, we get
$$a_jx_j=\dfrac{a_j}{\|(a_k)_{k=1}^\infty\|_\infty}\|(a_k)_{k=1}^\infty\|_\infty x_j \in V \subseteq U$$
  for every $j \geq j_0$. It follows that 
  $a_jx_j \stackrel{\tau_E}{\longrightarrow} 0$, from which we obtain  $(a_jx_j)_{j=1}^\infty \in A$, proving that $A$ is $\ell_\infty$-complete.

  \medskip

\noindent (b) Again, by Theorem \ref{main} it is enough to check that the set
$$A:=\{(x_j)_{j=1}^\infty \in \ell_\infty (E) \colon  x_j \stackrel{\tau_E}{\longrightarrow} 0 \textrm{ and } f(x_j) \stackrel{\tau}{\longrightarrow} 0 \} $$
is $\ell_\infty$-complete and subsequence invariant. Subsequence invariance is obvious, and $\ell_\infty$-complete-ness follows from  an adaptation of the reasoning used in the proof of (a)  as follows: to prove that  $f(a_j x_j) \stackrel{\tau}{\longrightarrow} 0 $, use that $\tau$ is also a vector topology, let $n$ be the natural number that works in the definition of map of homogeneous type for $f$ and work with the  convergence $\|(a_k)_{k=1}^\infty\|_\infty^n f(x_j) \stackrel{\tau}{ \longrightarrow} 0$ instead of $\|(a_k)_{k=1}^\infty\|_\infty f(x_j) \stackrel{\tau}{ \longrightarrow} 0$. 
\end{proof}

\section{Applications to Banach spaces}

In this section we give applications of our results to Banach spaces failing well studied properties and to operators/homogeneous polynomials not belonging to well studied classes. Linear operators and homogeneous polynomials in this section are always continuous and act between Banach spaces. We shall use (without warning) that linear operators and homogeneous polynomials are maps of homogeneous type and that the weak topology on a Banach space and the weak$^*$ topology on a dual Banach spaces are vector topologies weaker than the norm topology. We begin improving some known results.


A sequence $(x_j)_{j=1}^\infty$ in a Banach space $E$ is {\it polynomially null} if $P(x_j) \longrightarrow 0$ for every scalar-valued homogeneous polynomial $P$ on $E$.

$\bullet$ A Banach space space $E$ has the {\it polynomial Schur property} if polynomially null sequence in $E$ are norm null. Such spaces are also called {\it $\Lambda$-spaces}. This property was introduced by Carne, Cole and Gamelin in \cite{Carne}, further developments  can be found in, eg., \cite{Arraz, JLucasPAMS, JLPL, Farmer, Jaramillo}.

$\bullet$ A bounded linear operator is {\it completely continuous} if it sends weakly null sequences to norm null sequences. A continuous homogeneous polynomial is {\it weakly sequentially continuous} if it sends weakly null sequences to norm null sequences. These notions are classic and the literature about them is vast.

\begin{corollary}\label{fcor} {\rm (a)} Let $T \colon E \longrightarrow F$ be a non-completely continuous operator between Banach spaces. Then the set of $E$-valued weakly null sequences $(x_j)_{j=1}^\infty$ such that $T(x_j) \not\longrightarrow 0$ is almost pointwise spaceable in $c_0^w(E)$.

Furthermore, if $E$ fails the Schur property and $\bf x$ is a seminormalized weakly null $E$-valued sequence, then there exists a closed infinite dimensional subspace of $c_0^w(E)$ consisting, up to the origin, of weakly null non-norm null sequences containing $\bf x$.\\
{\rm (b)} Let $P \colon E \longrightarrow F$ be a non-weakly sequentially continuous homogeneous polynomial. Then the set of $E$-valued weakly null sequences $(x_j)_{j=1}^\infty$ such that $P(x_j) \not\longrightarrow 0$ is almost pointwise spaceable in $c_0^w(E)$. \\
{\rm (c)} Let $E$ be a Banach space failing the polynomial Schur property. Then the set of $E$-valued polynomially null non-norm null sequences is almost pointwise spaceable in $c_0^w(E)$.

 Furthermore, if $\bf x$ is a seminormalized polynomially null $E$-valued sequence, then there exists a closed infinite dimensional subspace of $c_0^w(E)$ consisting, up to the origin, of polynomially null non-norm null sequences containing $\bf x$.
\end{corollary}

\begin{proof} Note that polynomially null sequences are weakly null, so everything in all three items happens within $c_0^w(E)$. Items (a) and (b) follow from Proposition \ref{prop4}. It is obvious that the set of $E$-valued polynomially null sequences is subsequence invariant. To obtain (c) from Theorem \ref{main} we just have to check that this set is $\ell_\infty$-complete. Indeed, given a scalar-valued $m$-homogeneous polynomial $P$ on $E$, a polynomially null sequence $(x_j)_{j=1}^\infty$ in $E$ and $(a_j)_{j=1}^\infty \in \ell_\infty$, $P(a_jx_j) = a_j^mP(x_j) \longrightarrow 0$ because $P(x_j) \longrightarrow 0$ and $(a_j^m)_{j=1}^\infty$ is bounded. The second statements of (a) and (c) follow from Corollary \ref{ncor}.
\end{proof}

Item (a) of the corollary above improves \cite[Corollary 2.10(a)]{michigan}. The second statement of (a) improves \cite[Theorem 2.1]{Jimenez}, which is the original source of the technique we used to prove our main result. Item (b) improves \cite[Corollary 2.10(b)]{michigan} and item (c) improves \cite[Theorem 2.1(a)]{JLPL}.

\begin{examples}\rm Non-completely continuous operators and non-weakly sequentially continuous homogeneous polynomials are abundant in the literature. The following spaces fail the polynomial Schur property: $L_1[0,1], \ell_\infty, c_0$ and Banach spaces containing a copy of $c_0$ (see \cite[Theorems 6.2, 6.5 and 7.5]{Carne}).
\end{examples}

From now on, we shall give applications of our results to situations that, to the best of our knowledge, have not been handled in the context of lineability thus far.
\begin{corollary} Let $E$ be a separable Banach space and let $T \colon E \longrightarrow E$ be a non-completely continuous operator whose adjoint $T^*$ is completely continuous. Then, the set of weak$^*$-null sequences $(x_j^*)_{j=1}^\infty$ in  $E^*$ for which $T^*(x_j^*) \not\longrightarrow 0$ in $E^*$ is almost pointwise spaceable in $\ell_\infty(E^*)$.
\end{corollary}

\begin{proof} By \cite[Theorem 3]{Cima}, the set of weak$^*$-null sequences $(x_j^*)_{j=1}^\infty$ in  $E^*$ for which $T^*(x_j^*) \not\longrightarrow 0$ in $F$ is nonempty. Actually, there is a normalized weak$^*$-null sequence $(x_j^*)_{j=1}^\infty$ in  $E^*$ such that  $(T^*(x_j^*))_{j=1}^\infty$ is seminormalized. The result follows from Proposition \ref{prop4}.
\end{proof}

\begin{examples}\rm Let $E$ be an infinite dimensional separable Banach space such that $E^*$ has the Schur property. For instance, $E = c_0$ or $E = \left({\displaystyle \oplus_{j=1}^\infty} E_j\right)_0$, where each $E_j$ is  separable and $E_j^*$ has the Schur property. In particular, $E = \left(\displaystyle{\oplus_{j=1}^\infty} \ell_2^j\right)_0$ is a separable infinite dimensional Banach space whose dual $E^*= \left(\displaystyle{\oplus_{j=1}^\infty} \ell_2^j\right)_1$ is the Stegall space, which has the Schur property.   As an infinite dimensional Banach space and its dual cannot have the Schur property simultaneously (see \cite[Theorem 3]{thakare} or \cite[Theorem 3 and remark thereafter]{diestelDP}), $E$ fails the Schur property. Hence, there are plenty of non-completely continuous operators from $E$ to $E$, all of them having completely continuous adjoints because $E^*$ has the Schur property.
\end{examples}

For the identity operator on a dual Banach space, the next result is more general than the previous corollary.

\begin{corollary} Let $E$ be an infinite dimensional Banach space. Then the set of weak$^*$-null non-norm null sequences in $E^*$ is almost pointwise spaceable in $\ell_\infty(E^*)$.

Furthermore, given a  seminormalized weak$^*$ null $E^*$-valued sequence $\bf x^*$, there is a closed infinite dimensional subspace of $\ell_\infty(E^*)$ consisting, up to the origin, of weak$^*$-null non-norm null sequences and containing $\bf x^*$.
\end{corollary}

\begin{proof} The Josefon-Nissenzweig Theorem assures the existence of a normalized weak$^*$ null $E^*$-valued sequence. Using the identity operator on $E^*$, the first statement follows from Proposition \ref{prop4} and the second from Corollary \ref{ncor}.
\end{proof}

In the corollary above, given a closed infinite dimensional subspace $F$ of $E^*$, it is not always true that there is a weak$^*$-null non-norm null sequence belonging to $F$. The obvious example is $E = \ell_1$ and $F = c_0$. To remedy this situation it is enough to avoid copies of $\ell_1$:

\begin{corollary} Let $E$ be Banach space containing no copy of $\ell_1$ and let $F$ be a closed infinite dimensional subspace of $E^*$. Then the set
$${\cal C} = \{(x_j^*)_{j=1}^\infty \in \ell_\infty(F) \colon x_j^* \stackrel{\omega^*}{\longrightarrow} 0 \mbox{ in } E^*  ~\textrm{and}~ x_j^*\not \longrightarrow 0\} $$
is almost pointwise spaceable in $\ell_\infty(F)$.
\end{corollary}

\begin{proof} The fact that $\cal C$ is nonempty was established in \cite[Theorem 1(a)]{hagler} for real spaces and extended in \cite[Corollary 6]{mujicaarxiv} for complex spaces. Given a sequence ${\bf x^*} \in {\cal C}$, we have ${\bf x^*} \in \ell_\infty(E^*)$, so, using that the norm of $F$ is the norm of $E^*$, by Proposition \ref{prop4} there is a closed infinite dimensional subspace $V$ of $\ell_\infty(E^*)$ formed, up to the origin, by weak$^*$-null sequences non-norm null in $E^*$ and containing ${\bf x^*}$. Since ${\bf x^*} \in \ell_\infty(F)$, the proof of Theorem \ref{main} makes clear that $V$ is contained in $\ell_\infty(F)$, therefore $V \subseteq {\cal C}\cup\{0\}$.
\end{proof}

Let $1 \leq p < \infty$ be given. A sequence $(x_j)_{j=1}^\infty$ in a Banach space $E$ is {\it weakly $p$-summable} if $(x^*(x_j))_{j=1}^\infty \in \ell_p$ for every $x^* \in E^*$ (see \cite{diestel+jarchow+tonge}).

$\bullet$ A linear operator between Banach spaces is {\it $p$-convergent} if it sends weakly $p$-summable sequences to norm null sequences.

This class was introduced in \cite{castillo, castillo2}. The case $p = 1$ was studied, with a different definition, in \cite{gon, pel} under the name {\it unconditionally summing operators}, and, for arbitrary $1 \leq p < \infty$, under the name {\it unconditionally $p$-summing operators} in \cite{jm}. The equivalence of the definitions was proved in \cite[Theorem 1.7]{jm}. Recent developments can be found, e.g., in \cite{ardakani, ariel, chen}.

\begin{corollary} Let $T \colon E \longrightarrow F$ be non-$p$-convergent operator, $1 \leq p < \infty$. Then the set of weakly $p$-summable in $E$ such that $(T(x_j))_{j=1}^\infty$ is non-norm null in $F$ is almost pointwise spaceable in $c_0^w(E^*)$.
\end{corollary}

\begin{proof} It is obvious that the set of weakly $p$-summable sequences is a subsequence invariant subset of $c_0^w(E)$. Its $\ell_\infty$-completeness is easily checked. So, the result follows from Theorem \ref{main}.
\end{proof}

\begin{examples}\rm Let $(e_j)_{j=1}^\infty$ denote the sequence of canonical unit vectors in spaces of scalar-valued sequences. Then, $(e_j)_{j=1}^\infty$ is normalized and weakly 1-summable in $c_0$, and it is normalized and weakly $q$-summable in $\ell_q$ for $2 \leq q < \infty$. Therefore, the identity operator on $c_0$ fails to be $p$-convergent for every $1 \leq p < \infty$, and the identity operator on $\ell_q$, $2 \leq q < \infty$,  fails to be $p$-convergent for every $p\geq q$.
\end{examples}

A Banach space has the {\it Grothendieck property} if weak$^*$ null sequences in its dual are weakly null. This property was introduced by Grothendieck \cite{Grothendieck} in 1953 and has been developed since then. For a comprehensive account on the Grothendieck property we refer the reader to \cite{GonzalezKania}.

One more application of Proposition \ref{prop4} and Corollary \ref{ncor} gives the following:

\begin{corollary} Let $E$ be a Banach space failing the Grothendieck property. Then the set of weak$^*$-null non-weaklly null sequences in $E^*$ is almost pointwise spaceable in $\ell_\infty(E^*)$.

Furthermore, given a  seminormalized weak$^*$ null non-weakly null $E^*$-valued sequence $\bf x^*$, there is a closed infinite dimensional subspace of $\ell_\infty(E^*)$ consisting, up to the origin, of weak$^*$-null non-weakly null sequences and containing $\bf x^*$.
\end{corollary}

\begin{examples}\rm A separable Banach space has the Grothendieck property if and only if it is reflexive (see \cite[p.\,263]{GonzalezKania}). For being nonreflexive separable spaces, $c_0, \ell_1, L_1[0,1]$ and $C[0,1]$ fail the Grothendieck property.
\end{examples}

 In \cite{Peralta}, the authors consider a locally convex (hence vector) topology $\tau_R$ on a Banach $E$, called the {\it Right topology}, which is the smallest topology on $E$ so that the following equivalence holds: regardless of the Banach space $F$, a linear operator  $T \colon E \longrightarrow F$ is weakly compact if and only if $T \colon (E, \tau_R) \longrightarrow (F, \|\cdot\|)$ is continuous. It is clear that $\tau_R$ lies between the weak topology and the norm topology. 
 The following class of operators was also introduced in \cite{Peralta}, further developments can be found, e.g., in \cite{Ioana2, Ioana, Wen}.  

\medskip

\noindent $\bullet$  A bounded linear operator between Banach spaces is {\it pseudo weakly compact} if it maps $\tau_R$-null sequences in $E$ to norm null sequences in $F$.

These operators are sometimes called {\it Dunford-Pettis completely continuous opearators} (see \cite[Proposition 1]{Ioana}).

\begin{corollary} Let $T \colon E \longrightarrow F$ be a non-pseudo weakly compact operator. Then the set of $\tau_R$-null sequences in $E$ such that $(T(x_j))_{j=1}^\infty$ is non-norm null in $F$ is almost pointwise spaceable in $c_0^w(E)$.
\end{corollary}

\begin{proof} Since $\tau_R$ is a vector topology weaker than the norm topology and $\tau_R$-null sequences are weakly null, because the weak topology is weaker than $\tau_R$, the result follows from Proposition \ref{prop4}.
\end{proof}

\begin{examples} \rm Recall that an operator $T \colon E \longrightarrow F$ between Banach spaces:\\
$\bullet$ Is {\it unconditionally converging} if, for every weakly unconditionally Cauchy series $\sum\limits_{j=1}^\infty x_j$ in $E$, the series $\sum\limits_{j=1}^\infty T(x_j)$ is unconditionally convergent in $F$. \\
$\bullet$ {\it Fixes a copy of $c_0$} if there exists a subspace $S$ of $E$ isomorphic to $c_0$ for which the restriction of $T$ to $S$ is an isomorphism onto $T(S)$.

In \cite[Proposition 14]{Peralta} it is proved that pseudo weakly compact operators are unconditionally converging; and, according to \cite[Exercise 8, p.\,64]{Diestel}, an operator fails to be unconditionally converging if and only if it fixes a copy of $c_0$. Therefore, every operator that fixes a copy of $c_0$ fails to be pseudo weakly compact. In particular, the identity operator on every Banach space containing a copy of $c_0$ is a non-pseudo weakly compact operator.
\end{examples}


\section{Applications to Banach lattices}

In this section we shall use freely that, for a Banach lattice $E$, $\ell_\infty(E)$ is a Banach lattice with the pointwise ordering \cite[p.\,183]{Ali1}. Unlike the previous sections, we shall not pass to $c_0^w(E)$ in this section because, in general, $c_0^w(E)$ is not a sublattice of $\ell_\infty(E)$ (see \cite[Remark 2.3(1)]{JLPL}). As usual, $E^+$ denotes the positive cone of the Banach lattice $E$. By a {\it Banach sequence lattice} we mean a linear subspace $X$ of $E^\mathbb{N}$, where $E$ is a Banach lattice, endowed with a complete norm that makes $X$ a Banach lattice with the pointwise ordering.

\begin{definition}\rm Let $A$ be a subset of a Banach sequence lattice $X$ such that $A \cap X^+ \neq \emptyset$. We say that $A$  is {\it almost positive pointwise latticeable} if, for each positive sequence $ {\bf x^*}\in A$, there exists a closed infinite dimensional sublattice of $X$ contained in $A \cup \{0\}$ and containing a subsequence of ${\bf x^*}$.
\end{definition}

To be consistent with the terminology of \cite{JLPL, oikhberg}, we should have written {\it almost positively pointwisely completely latticeable}, but, for simplicity, we shall write {\it almost positive pointwise latticeable}.

Again, it is worth mentioning that almost positive pointwise latticeability does not pass, automatically, from a set to its supersets.

\begin{theorem} \label{ret} Let $f \colon E \longrightarrow F$ be a map of homogeneous type from a Banach lattice $E$ to a Banach space $F$, let $A$ be a subsequence invariant and $\ell_\infty$-complete subset of $\ell_\infty(E)$, and let  $\tau$ be a vector topology on $F$ weaker than the norm topology. Consider the sets
$$ \mathcal{C}_1 =\{(x_j)_{j=1}^\infty \in A \colon   f(x_j){\stackrel{\tau}{\not\longrightarrow}} 0\}~,~\mathcal{C}_2 =\{(x_j)_{j=1}^\infty \in A \colon (x_j)_{j=1}^\infty ~\textrm{is disjoint and}~  f(x_j){\stackrel{\tau}{\not\longrightarrow}} 0\}.$$
For $i=1,2$,  $\mathcal{C}_i\cap \ell_\infty(E)^{+} = \emptyset$ or $C_i$ is almost positive pointwise latticeable in $\ell_\infty(E)$.
\end{theorem}

\begin{proof} Suppose that $\mathcal{C}_1\cap \ell_\infty(E)^{+} \neq \emptyset$ and let ${\bf x} =(x_j)_{j=1}^\infty$ be a positive sequence in $A$ such that $f(x_j){\stackrel{\tau}{\not\longrightarrow}} 0.$ Considering the operator  $T\colon \ell_\infty \longrightarrow \ell_\infty(E)$ from the proof of Theorem \ref{main}, we know that $T(\ell_\infty)$ is a subspace of $\ell_\infty(E)$ isomorphic to $\ell_\infty$, contained in $\mathcal{C}_1 \cup \{0\}$ and containing a subsequence of ${\bf x}$. All that is left to prove is that $T(\ell_\infty)$ is a sublattice of $\ell_\infty(E)$. To do so, it is enough to check that $T$ is a Riesz homomorphism. For every sequence ${\bf a} =(a_j)_{j=1}^\infty \in \ell_\infty$, using that $x_j \geq 0$ for every $j$ and the definition of $T$ by means of the surjective map $G \colon \mathbb{N} - \{1\} \longrightarrow \mathbb{N}$, we get
\begin{align*}
|T({\bf a})|& = |T((a_j)_{j=1}^\infty)|=|(a_{G(k+1)}x_k)_{k=1}^\infty|=(|a_{G(k+1)}x_k|)_{k=1}^\infty=(|a_{G(k+1)}|x_k)_{k=1}^\infty\\
&=T((|a_k|)_{k=1}^\infty)=T(|(a_k)_{k=1}^\infty|) = T(|{\bf a}|).
\end{align*}
This proves that $T$ is a Riesz homomorphism and completes the proof for $C_1$.

Suppose now that $\mathcal{C}_2\cap \ell_\infty(E)^{+} \neq \emptyset$ and let ${\bf x} =(x_j)_{j=1}^\infty$ be a positive disjoint sequence in ${\cal C}_2$. From the case of ${\cal C}_1$ we know that $T(\ell_\infty)$ is a sublattice of $\ell_\infty(E)$ isomorphic to $\ell_\infty$, contained in $\mathcal{C}_1 \cup \{0\}$ and containing a subsequence of ${\bf x}$. We just have to prove that $T(\ell_\infty) \subseteq {\cal C}_2 \cup \{0\}$, that is, we have to prove that each sequence of $T(\ell_\infty)$ is disjoint. Given ${\bf a} =(a_j)_{j=1}^\infty \in \ell_\infty$, since $x_j \perp x_k$ for all $j \neq k$, we have $a_{G(j+1)}x_j \perp a_{G(k+1)}x_k$ for all $j \neq k$ by \cite[Lemma 1.9]{Ali}. Then the sequence
$T({\bf a}) = (a_{G(j+1)}x_j)_{j=1}^\infty$ is disjoint. \end{proof}

Reasoning as we did before Corollary \ref{ncor}, we get the following.

\begin{corollary}\label{nncor} Let $E, F, f, A, \tau, {\cal C}_1$ and ${\cal C}_2$ be as in Theorem {\rm \ref{ret}}. For $i = 1,2$, if there is a seminormalized sequence $\bf x$ belonging to ${\cal C}_i$, then there exists a closed infinite dimensional sublattice of $\ell_\infty(E)$ contained in ${\cal C}_i \cup \{0\}$ and containing $\bf x$.
\end{corollary}

As to applications of the results above, we start with Banach lattices not enjoying certain well studied properties.

\begin{corollary}\label{corroc} For any Banach lattice $E$, each of the sets 
$$\mathcal{C}_1=\{(x_j)_{j=1}^\infty \in E^{\mathbb{N}} \colon (x_j)_{j=1}^\infty ~\textrm{is order bounded, disjoint and}~x_j \not\longrightarrow 0 \} \mbox{ and}$$
$$\mathcal{C}_2=\{(x_j)_{j=1}^\infty \in E^{\mathbb{N}} \colon (x_j)_{j=1}^\infty ~\textrm{is norm bounded, disjoint and}~x_j \stackrel{\omega}{\not\longrightarrow} 0 \}$$
is empty or almost positive pointwise latticeable in $\ell_\infty(E)$.
\end{corollary}

\begin{proof} It is obvious that ${\cal C}_1 \cup {\cal C}_2 \subseteq \ell_\infty(E)$. Let us see that ${\cal C}_i \neq \emptyset \Longrightarrow {\cal C}_i \cap \ell_\infty(E)^+ \neq \emptyset$. Indeed, given $(x_j)_{j=1}^\infty \in {\cal C}_i$, it is obvious that the sequence $(|x_j|)_{j=1}^\infty$ is disjoint and order bounded (if $a \leq x_j \leq b$ for every $j$, then $0 \leq |x_j| \leq b \vee (-a)$ for every $n$). On the one hand, it is also obvious that, if $x_j \not\longrightarrow 0$, then $|x_j| \not\longrightarrow 0$. On the other hand, it is not difficult to check that, if $x_j \stackrel{\omega}{\not\longrightarrow} 0$, then $|x_j| \stackrel{\omega}{\not\longrightarrow} 0$. Therefore, $(|x_j|)_{j=1}^\infty \in {\cal C}_i\cap \ell_\infty(E)^+ $.

From Theorem \ref{ret}, the case of ${\cal C}_2$ follows immediately; and for ${\cal C}_1$ it suffices to show that the set of order bounded $E$-valued sequences is subsequence invariant and $\ell_\infty$-complete. Subsequence invariance is straighforward. Let $(x_j)_{j=1}^\infty$ be an order bounded $E$-valued sequence, say, $z \leq x_j \leq w$ for every $j$, and let $(a_j)_{j=1}^\infty \in \ell_\infty$ be given, say, $c_1 \leq a_j \leq c_2$ for every $j$. If $a_j \geq 0$, then $c_1z \leq a_j x_j \leq c_2 w$. If $a_j < 0$, then $-c_2 w \leq a_j x_j \leq -c_1z$. Hence,
 $$c_1a \wedge (-c_2b) \leq a_j x_j \leq (-c_1a) \vee c_2 b  $$
for every $j$, proving that the set of order bounded sequences is $\ell_\infty$-complete.
\end{proof}

\begin{examples}\rm (1) For every Banach lattice $E$ not having order continuous norm, the set ${\cal C}_1$ of the corollary above is nonempty by \cite[Theorem 4.14]{Ali1}, thus the set of non-norm null, order bounded and disjoint $E$-valued sequences is almost positive pointwise latticeable in $\ell_\infty(E)$.\\
(2) For every Banach lattice $E$ such that $E^*$ does not have order continuous norm, the set ${\cal C}_2$ of the corollary above is nonempty by \cite[Theorem 4.69]{Ali1}, thus the set of non-weakly null, norm bounded and disjoint $E$-valued sequences is almost positive pointwise latticeable in $\ell_\infty(E)$.
\end{examples}

On the one hand, it is known that order bounded disjoint sequences in Banach lattices are weakly null \cite[p.\! 192]{Ali1}; on the other hand, the following property was introduced in \cite{Omid} in order to study reflexive Banach lattices:

\noindent$\bullet$ A Banach lattice has the {\it disjoint Grothendieck property} if norm bounded disjoint sequences in its dual are weakly null.

\begin{corollary} Let $E$ be a Banach lattice failing the disjoint Grothendieck property. Then the set of norm bounded disjoint non-weakly null $E^*$-valued sequences is almost positive pointwise latticeable in $\ell_\infty(E^*)$.
\end{corollary}

\begin{proof} The failure of the disjoint Grothendieck property gives a norm bounded disjoint non-weakly null sequence $(x_j^*)_{j=1}^\infty$ in $E^*$. The same reasoning of the proof of Corollary \ref{corroc} shows that $(|x_j^*|)_{j=1}^\infty$ is a positive norm bounded disjoint non-weakly null sequence in $E^*$. The result follows from Corollary \ref{corroc}.
\end{proof}

\begin{examples}\rm $\ell_\infty$ and Banach lattices containing a copy of $\ell_\infty$ fail the disjoint Grothendieck property (see \cite[Theorem 4.56]{Ali1} and \cite[p.\,4]{Omid}).
\end{examples}

Next we state the counterpart of Proposition \ref{prop4} to Banach lattices. The argument is essentially the same, with Theorem \ref{ret} playing the role of Theorem \ref{main}, so we skip the proof.

\begin{proposition}\label{ret2} Let $E$ be a Banach lattice, let $F$ be a Banach space, let $f \colon E \longrightarrow F$ be a map of homogeneous type, and let $\tau_E, \tau_F$ be vector topologies in $E$ and $F$, respectively, with $\tau_F$ weaker than the norm topology. Consider the sets
    $$ {\cal C}_1=
    \{(x_j)_{j=1}^\infty \in \ell_\infty(E) \colon x_j \stackrel{\tau_E}{\longrightarrow 0} \textrm{and } f(x_j)\stackrel{\tau_F}{\not \longrightarrow 0}\} \mbox{ and}$$
$$ {\cal C}_2=
    \{(x_j)_{j=1}^\infty \in \ell_\infty(E) : (x_j)_{j=1}^\infty \mbox{ is disjoint}, x_j \stackrel{\tau_E}{\longrightarrow 0} \textrm{and } f(x_j)\stackrel{\tau_F}{\not \longrightarrow 0}\}.$$
For $i = 1,2$, if ${\cal C}_i$ contains a positive sequence, then it is almost positive pointwise latticeable. 
\end{proposition}

The following property, introduced by Wnuk \cite{Wnuk1, Wnuk2} and R\"{a}biger \cite{Rabiger}, is a quite popular topic in Banach lattice theory; in particular it is the subject of a number of recent papers, see, e.g., \cite{ardvin, ardval, bumonat, GJV, chen2023} and references therein.

\noindent $\bullet$ A Banach lattice has the {\it positive Schur property} if positive (or disjoint or positive disjoint) weakly null sequences are norm null.

\medskip

For a Banach lattice $E$ failing the positive Schur property, the existence of a closed infinite dimensional sublattice of $\ell_\infty(E)$ consisting, up to $0$, of disjoint non-norm null weakly null $E$-valued sequences, was established in \cite[Theorem 2.1(b)]{JLPL}. In order to improve this result, let us recall the definition of the absolute weak topology:

Given a Banach lattice $E$ and a functional $x^* \in E^*$, consider the seminorm $x \in E \mapsto p_{x^*}(x) := |x^*|(|x|)$. The {\it absolute weak topology} is the locally convex-solid topology $|\sigma|(E,E^*)$ on $E$ generated by the family of seminorms $\{p_{x^*} : x^* \in E^*\}$, see \cite[p.\,172]{Ali1}. As a locally convex topology, the absolute weak topology is a vector topology, and it is clear that it lies between the weak and norm topologies. A sequence that converges to $0$ with respect to this topology shall be called an {\it absolutely weakly null} sequence, in symbols, $x_j{\stackrel{|\sigma|}{\longrightarrow}} 0.$

\begin{corollary} If the Banach lattice $E$ fails the positive Schur property, then the sets $$\mathcal{C}_1=\{(x_j)_{j=1}^\infty \in \ell_\infty(E) \colon x_j{\stackrel{|\sigma|}{\longrightarrow}} 0 ~\textrm{and}~ x_j \not\longrightarrow 0\}~ \mbox{and}$$
$$\mathcal{C}_2=\{(x_j)_{j=1}^\infty \in \ell_\infty(E) \colon (x_j)_{j=1}^\infty ~\textrm{is disjoint,}~ x_j{\stackrel{|\sigma|}{\longrightarrow} 0 ~\textrm{and}~ x_j \not\longrightarrow} 0\}$$
are almost positive pointwise latticeable.
\end{corollary}
\begin{proof} As $E$ fails the positive Schur property, there exists a positive disjoint weakly null non-norm null sequence $(x_j)_{j=1}^\infty$ in $E$. Since a positive sequence is absolutely weakly null if and only if it is weakly null, $(x_j)_{j=1}^\infty$ is absolutely weakly null. Hence, $\mathcal{C}_1 \cap \ell_\infty (E)^+ \neq \emptyset \neq \mathcal{C}_2 \cap \ell_\infty (E)^+ $. 
Proposition \ref{ret2} yields the result. 
\end{proof}

\begin{examples}\rm The following are examples of Banach lattices failing the positive Schur property: (i) Reflexive infinite dimensional spaces, in particular, for $1 < p < \infty$, $\ell_p(\Gamma)$ for $\Gamma$ infinite and infinite dimensional $L_p(\mu)$-spaces; (ii) Infinite dimensional $AM$-spaces, in particular, $c_0(\Gamma), \ell_\infty(\Gamma)$ for $\Gamma$ infinite, and $C(K)$ for any infinite compact Hausdorff space; (iii) Infinite dimensional Banach lattices not containing a lattice copy of $\ell_1$.
\end{examples}

A Banach lattice $E$ has the:\\
$\bullet$ {\it Dual positive Schur property} if positive disjoint weak$^*$ null sequences in $E^*$ are norm null.\\
$\bullet$ {\it Positive Grothendieck property} if positive weak$^*$ null sequences in $E^*$ are weakly null.

These two properties were introduced by Wnuk \cite{Wnuk}, where he remarked that a Banach lattice $E$ has the dual positive Schur property if and only if $E$ has the positive Grothendieck property and $E^*$ has the positive Schur property. Recent developments can be found in \cite{GJV, bu, Galindo}.

The absolute weak$^*$ topology on the dual $E^*$ of a Banach lattice $E$ is the locally convex-solid topology on $E^*$ generated by the family of seminorms  $\{q_x: x \in E\}$, where $q_x(x^*) = |x^*|(|x|)$  for $x^* \in E^*$ and $x \in E$ \cite[p.\,176]{Ali1}. This topology lies between the weak$^*$ and the weak topologies on $E^*$.

\begin{corollary} Let $E$ be a Banach lattice.\\
{\rm (a)} If $E$ fails the positive Grothendieck property, then the set of weak$^*$ null non weakly null $E^*$-valued sequences is almost positive pointwise latticeable in $\ell_\infty(E^*)$. \\
{\rm (b)} If $E$ fails the dual positive Schur property, then the set of weak$^*$ null (or absolutely weak$^*$-null) disjoint non-norm null $E^*$-valued sequences is almost positive pointwise latticeable in $\ell_\infty(E^*)$.
\end{corollary}

\begin{proof} Each item follows from the definition of the corresponding property and Proposition \ref{ret2}. For the part of the absolute weak$^*$-null topology on (b), use that positive weak$^*$-null sequences are absolutely weak$^*$-null.
\end{proof}

\begin{examples}\rm Infinite dimensional $AL$-spaces fail the positive Grothendieck property \cite[p.\,764]{Wnuk}. Reasoning with the canonical unit vectors, it is easy to see that $c_0$ fails the dual positive Schur property. More generally, if there exists a positive non-compact operator from a Banach lattice $E$ to $c_0$, then $E$ fails the dual positive Schur property \cite[Proposition 2.7]{Wnuk}.
\end{examples}

We shall finish the paper by showing that the results of this section can be applied to operators not belonging to several already studied classes. We do not intend to be exhaustive, we just want to illustrate that the results apply to classical classes of operators, to classes that have been explored by many experts and to very recently introduced classes. References to each of these classes shall be given, so examples of operators not belonging to each of these classes can be found in the corresponding references.

An operator $T \colon E \longrightarrow F$ from a Banach lattice $E$ to a Banach space $F$ is:

\noindent$\bullet$ {\it Order weakly compact} if $T([0,x])$ is relatively weakly compact in $F$ for every $x \in E^+$
. This class, introduced by Dodds in 1975, is a classic topic in Banach lattice theory, see \cite[Chapter 5]{Ali1} and \cite[Section 3.4]{Meyer}.

\noindent$\bullet$ {\it $M$-weakly compact} if $T$ maps norm bounded disjoint sequences in $E$ to norm null sequences in $F$. This class, introduced by Meyer-Nieberg in 1974, is also a classic topic in Banach lattice theory, see \cite[Chapter 5]{Ali1} and \cite[Section 3.6]{Meyer}.

\noindent$\bullet$ {\it Almost Dunford-Pettis} if $T$ maps disjoint weakly null sequences in $E$ to norm null sequences in $F$. This class was introduced by S\'anchez \cite{sanchez}, significant developments appeared in \cite{aqzzouz}, for recent contributions,  see \cite{ardvin, ardval, bumonat, khabaoui, Moussa}.

\noindent$\bullet$ {\it Weak $M$ weakly compact} if $T$ maps disjoint norm bounded sequences in $E$ to weakly null sequences in $F$. This class was recently introduced in \cite{Moussa} and a development appeared in \cite{Alpay}.

\begin{corollary} Let $T \colon E \longrightarrow F$ be an operator from a Banach lattice $E$ to a Banach space $F$.\\
{\rm (a)} If $T$ fails to be order weakly compact, then the set of order bounded weakly null sequences $(x_j)_{j=1}^\infty$ in $E$ such that $T(x_j) \not\longrightarrow 0$ in $F$ is almost positive pointwise latticeable in $\ell_\infty(E)$. \\
{\rm (b)} If $T$ fails to be almost Dunford-Pettis, then the set of disjoint weakly null sequences $(x_j)_{j=1}^\infty$ in $E$ such that $T(x_j) \not\longrightarrow 0$ in $F$ is almost positive pointwise latticeable in $\ell_\infty(E)$.
\end{corollary}

\begin{proof} (a) As $T$ is not order weakly compact, by \cite[Ex.\,3, p.\,336]{Ali1} there exists a positive order bounded weakly null sequence $(x_j)_{j=1}^\infty$ in $E$ such that $T(x_j) \not\longrightarrow 0$. As the set of order bounded sequences is subsequence invariant and $\ell_\infty$-complete (see the proof of Corollary \ref{corroc}), the result follows from Theorem \ref{ret}. \\
(b) As $T$ is not almost Dunford-Pettis, by \cite[Theorem 2.2]{aqzzouz} there exists a positive disjoint weakly null sequence $(x_j)_{j=1}^\infty$ in $E$ such that $T(x_j) \not\longrightarrow 0$. The result follows from Proposition \ref{ret2}.
\end{proof}

\begin{corollary} Let $T \colon E \longrightarrow F$ be a positive operator between Banach lattices.\\
{\rm (a)} If $T$ fails to be $M$-weakly compact, then the set of norm bounded disjoint sequences $(x_j)_{j=1}^\infty$ in $E$ such that $T(x_j) \not\longrightarrow 0$ in $F$ is almost positive pointwise latticeable in $\ell_\infty(E)$. \\
{\rm (b)} If $T$ fails to be weak $M$ weakly compact, then the set of disjoint bounded $(x_j)_{j=1}^\infty$ in $E$ such that $T(x_j) \stackrel{\omega}{\not\longrightarrow} 0$ in $F$ is almost positive pointwise latticeable in $\ell_\infty(E)$.
\end{corollary}

\begin{proof} (a) Since $T$ is not $M$-weakly compact, there is a norm bounded disjoint sequence $(x_j)_{j=1}^\infty$ in $E$ such that $T(x_j) \not\longrightarrow 0$. It is clear that $(|x_j|)_{j=1}^\infty$ is a positive norm bounded disjoint sequence in $E$. Using that $T$ is positive and $T(|x_j|) \not\longrightarrow 0$, the inequality $|T(x_j)| \leq T(|x_j|)$, which implies $\|T(x_j)\| \leq \| T(|x_j|)\|$, gives $T(|x_j|) \not\longrightarrow 0$. The result follows from Proposition \ref{ret2}.\\
(b) The proof is almost the same as in (a), the only difference is that we have to use that $T(x_j) \stackrel{\omega}{\not\longrightarrow} 0$ implies $T(|x_j|) \stackrel{\omega}{\not\longrightarrow} 0$; but this is also an easy consequence of  the positivity of $T$.
\end{proof}

\bigskip
\noindent Mikaela Aires~~~~~~~~~~~~~~~~~~~~~~~~~~~~~~~~~~~~~~~~~~~~~~Geraldo Botelho~\\
Instituto de Matem\'atica e Estat\'istica~~~~~~~~~~~~~~~Instituto de  Matem\'atica e Estat\'istica\\
Universidade de S\~ao Paulo~~~~~~~~~~~~~~~~~~~~~~~~~~~~~\hspace*{0,1em}Universidade Federal de Uberl\^andia\\
05.508-090 -- S\~ao Paulo -- Brazil~~~~~~~~~~~~~~~~~~~~~~\,\hspace*{0,1em}38.400-902 -- Uberl\^andia -- Brazil\\
e-mail: mikaela\_aires@ime.usp.br~~~~~~~~~~~~~~~~~~~~~\,e-mail: botelho@ufu.br
\bigskip


%

%
%
%

\end{document}